\documentclass[10pt]{amsart}
\usepackage{amsthm,amssymb,amsmath}

\theoremstyle{definition}

\newtheorem{proposition}{Proposition}
\newtheorem{theorem}{Theorem}
\newtheorem{lemma}{Lemma}

\newtheorem{remark}{Remark}

\newcommand{\esssup}{{\rm esssup}}
\newcommand{\loc}{{\rm loc}}
\newcommand{\supp}{{\rm supp}}

\newcommand{\Real}{\mbox{Re}}

\newcommand{\weak}{\rm {weak}}
\newcommand{\strong}{\rm {str}}
\newcommand{\comp}{{\rm com}}

\begin{document}

\title[Unique continuation]{Schr\"{o}dinger operators and unique continuation. Towards an optimal result}
\author{D.~Kinzebulatov}

\address{Department of Mathematics, University of Toronto, Toronto, Ontario, Canada M5S 2E4}

\email{dkinz@math.toronto.edu}

\author{L.~Shartser}

\address{Department of Mathematics, University of Toronto, Toronto, Ontario, Canada M5S 2E4}

\email{shartl@math.toronto.edu}



\begin{abstract}
In this article we prove the property of unique continuation (also known for $C^\infty$ functions as quasianalyticity) for solutions of the differential inequality $|\Delta u| \leq |Vu|$ for $V$ from a wide class of potentials (including $L^{d/2,\infty}_{\loc}(\mathbb R^d)$ class)
and 
$u$ in a space of solutions $Y_V$ containing all eigenfunctions of the corresponding self-adjoint Schr\"{o}dinger operator.

Motivating question: is it true that for potentials $V$, for which self-adjoint Schr\"{o}dinger operator is well defined, the property of unique continuation holds?
\end{abstract}

\maketitle

\section{Introduction}

Let $\Omega$ be an open set in $\mathbb R^d$ ($d \geq 3$),
$X_p:=L_{\loc}^p(\Omega,dx)$ ($p \geq 1$), $H^{m,p}(\Omega)$ the standard Sobolev space and $-\Delta:=-\sum_{k=1}^d \frac{\partial^2}{\partial x_k^2}$ the Laplace operator. Let
$\mathcal D'(\Omega)$ be the space of distributions on $\Omega$ and $\mathcal L_{\loc}^{2,1}(\Omega):=\{f \in X_1: \Delta f \in \mathcal D'(\Omega) \cap X_1\}$.

Let now $\Omega$ be connected. For $Y_V \subset \mathcal L_{\loc}^{2,1}(\Omega)$ a space of functions depending on $V \in X_1$ we say that the differential inequality
\begin{equation}
\label{diffineq}
|\Delta u(x)| \leq |V(x)||u(x)| \quad \text{ a.e. in } \Omega
\end{equation}
has the property of \textit{weak} unique continuation (WUC) in $Y_V$~($=:Y_V^{\weak}$) provided that whenever $u$ in $Y_V$ satisfies inequality (\ref{diffineq}) and vanishes in an open subset of $\Omega$ it follows that $u \equiv 0$ in $\Omega$.
We also say that (\ref{diffineq}) has the property of \textit{strong} unique continuation (SUC) in $Y_V$~($=:Y_V^{\strong}$) if whenever $u$ in $Y_V$ satisfies (\ref{diffineq})
and vanishes to an infinite order at a point $x_0 \in \Omega$, i.e.,
\begin{equation*}
\lim_{\rho \to 0} \frac{1}{\rho^k} \int_{|x-x_0|<\rho} |u(x)|^2 dx=0, \text{ for all } k \in \mathbb N,
\end{equation*}
it follows that $u \equiv 0$ in $\Omega$.

Throughout our work we make use of the following notations.  $\mathbf{1}_{S}$ is the characteristic function of a set $S \subset \mathbb R^d$, $B(x_0,\rho):=\{x \in \mathbb R^d:|x-x_0|<\rho\}$ and $B_S(x_0,\rho):=B(x_0,\rho) \cap S$ (also, set $B(\rho):=B(0,\rho)$ and $B_S(\rho):=B_S(0,\rho)$), $\|A\|_{p \mapsto q}$ is the norm of operator $A:L^p(\mathbb R^d) \mapsto L^q(\mathbb R^d)$, $(-\Delta)^{-\frac{z}{2}}$, $0<\Real(z)<d$, stands for the Riesz operator whose action on a function $f \in C_0^\infty(\mathbb R^d)$ is determined by the formula
\begin{equation*}
(-\Delta)^{-\frac{z}{2}} f(x)=c_z\int_{\mathbb R^d}\left[(-\Delta)^{-\frac{z}{2}}\right](x,y) f(y)dy,
\end{equation*}
where 
\begin{equation*}
\left[(-\Delta)^{-\frac{z}{2}}\right](x,y):=|x-y|^{z-d}, \quad c_z:=\Gamma\left(\frac{d-z}{2}\right)\left(\pi^{d/2} 2^z \Gamma \left(\frac{z}{2} \right) \right)^{-1}
\end{equation*}
(see, e.g., \cite{SteBook}).

The first result on unique continuation was obtained by T.~Carleman \cite{Carl}. He proved that (\ref{diffineq}) has the WUC property in the case $d=2$, $V \in L^{\infty}_{\loc}(\Omega)$.
Since then, 
the properties of
unique continuation were extensively studied by many authors (primarily following the original Carleman's approach), with
the best possible for $L_{\loc}^p$-potentials SUC result obtained by D.~Jerison and C.~Kenig ($p=\frac{d}{2}$, $Y_V^{\strong}=H^{2,\bar{p}}_{\loc}$, $\bar{p}:=\frac{2d}{d+2}$) \cite{JK}, and its extension for $L_{\loc}^{d/2,\infty}$-potentials obtained by E.M.~Stein \cite{Ste}. 
In \cite{SS,Forese,Saw} the authors 
are proving unique continuation for potentials from 
the following `abstract' classes:

1) $V \in L^2_{\loc}(\mathbb R^d)$ satisfies $\|\mathbf{1}_{B(x,1)}V(-\Delta)^{-1}|V| \mathbf{1}_{B(x,1)}\|_{2 \mapsto 2}<\infty$ for all $B(x,1) \subset \mathbb R^d$ in \cite{SS};

2) $V \in L^2_{\loc}(\mathbb R^d)$ satisfies $\inf_{\lambda>0}\|~V(-\Delta+\lambda)^{-1}\|_{2 \mapsto 2}=0$ and $$\inf_{\lambda>0}\sup_{x \in \mathbb R^d} \|\mathbf{1}_{B(x,1)}V_1(-\Delta+\lambda)^{-3/4}\|_{2 \mapsto 2}=0$$ for all $B(x,1) \subset \mathbb R^d$, see \cite{Forese};

3) Kato class (see Section \ref{compsect})
in \cite{Saw} ($d=3$).

Further improvements of Stein's result were obtained in \cite{SawChan,RV,W}, where unique continuation is proved for potentials $V$ locally in
Campanato-Morrey class (see Section \ref{compsect}).

Our main result is that 
differential inequality (\ref{diffineq}) has the WUC property in the space of solutions
% $Y_V^{\weak}$ and $Y_V^{\strong}$
$$Y_V^{\weak}:=\left\{f \in \mathcal L^{2,1}_{\loc}: |V|^{\frac{1}{2}}f \in X_2 \right\}$$
(containing eigenfunctions of the corresponding self-adjoint Schr\"{o}dinger operator, see below)
and, respectively, the SUC property in
$$Y_V^{\strong}:=Y_V^{\weak} \cap H_{\loc}^{1,\bar{p}}(\Omega), $$
%\quad \bar{p}:=\frac{2d}{d+2},$$
for potentials $V$ in the following class (for the motivation see (\ref{formineq99}) and (\ref{fbetaloc}) below)
\begin{equation*}
\mathcal F^d_{\beta,\loc}:=\left\{W \in X_{\frac{d-1}{2}}:
\sup_{K}\overline{\lim\limits_{\rho \to 0}} \sup_{x_0 \in K} \|\mathbf{1}_{B_K(x_0,\rho)} |W|^{\frac{d-1}{4}}(-\Delta)^{-\frac{d-1}{2}}|W|^{\frac{d-1}{4}}\mathbf{1}_{B_K(x_0,\rho)}\|_{2 \mapsto 2}\leq \beta\right\},
\end{equation*}
where $K$ is a compact subset of $\Omega$.

Historically, the most important reason for establishing the WUC property is its application to
the problem of absence of positive eigenvalues of the self-adjoint Schr\"{o}dinger operators,
discovered in 1959 by T.~Kato \cite{Kato1}. He proved
%used the WUC property to show the absence of positive eigenvalues of the self-adjoint Schr\"{o}dinger operators \cite{Kato1}, proving 
that if $V$ has a compact support, then all eigenfunctions corresponding to positive eigenvalues must vanish outside of a ball of finite radius, hence by WUC must be identically equal to zero. In what follows, we employ our WUC result for (\ref{diffineq}) to prove the absence of positive eigenvalues of the self-adjoint Schr\"{o}dinger operator $H \supset -\Delta+V$ in the complex Hilbert space $\mathcal H:=L^2(\mathbb R^d,dx)$ defined in the sense of quadratic forms (see \cite{Kato2,RS}), namely:
\begin{equation}
\label{formsum}
H:=H_+\dotplus(-V_-),
\end{equation}
where $H_+:=H_0\dotplus V_+$, $H_0=(-\Delta|_{C^\infty(\mathbb R^d)})^{\ast}$, $D(H_0)=H^{2,2}(\mathbb R^d)$, $V=V_+-V_-$, $V_{\pm} \geq 0$, $V_{\pm} \in L^1(\mathbb R^d)$ and
\begin{equation}
\label{formineq}
\inf_{\lambda>0}\left\|V_-^{\frac{1}{2}}(H_++\lambda)^{-1} V_-^{\frac{1}{2}}\right\|_{2 \mapsto 2} \leq \beta<1.
\end{equation}
The latter inequality guarantees the existence of the form sum (\ref{formsum}), see \cite[Ch.VI]{Kato2}),
and the inclusion $D(H) \subset Y_V^{\weak}$
(see Section \ref{comparisonsect}). 
On the other hand it is easy to see that if $V \in L_{\loc}^1(\mathbb R^d)$ satisfies the inequality
\begin{equation}
\label{formineq99}
\inf_{\lambda>0}\left\| |V|^{\frac{1}{2}}(H_0+\lambda)^{-1}|V|^{\frac{1}{2}}\right\|_{2 \mapsto 2} \leq \beta<1,
\end{equation}
then $V$ satisfies (\ref{formineq}) with the same $\beta$, and therefore the existence of the form sum (\ref{formsum}) follows.
The local nature of the problem of unique continuation for (\ref{diffineq}) leads to the definition of `local analogue' of potentials satisfying (\ref{formineq99}):
\begin{equation}
\label{fbetaloc}
F_{\beta,\loc}:=\left\{W \in X_1:\sup_K\overline{\lim\limits_{\rho \to 0}}\sup_{x_0 \in K}\|\mathbf{1}_{B_K(x_0,\rho)} |W|^{\frac{1}{2}}(-\Delta)^{-1}|W|^{\frac{1}{2}}\mathbf{1}_{B_K(x_0,\rho)}\|_{2 \mapsto 2} \leq \beta\right\}, 
\end{equation}
where $K$ is a compact subset of $\Omega$.
This class coincides with $\mathcal F^d_{\beta,\loc}$ if $d=3$, and contains $\mathcal F^d_{\beta,\loc}$ as a proper subclass if $d \geq 4$ (the latter easily follows from Heinz-Kato inequality, see, e.g., \cite{Kato0}). Arguments of this article do not apply to the larger class of potentials $F_{\beta,\loc}$ for $d \geq 4$.

\vspace*{4mm}

Class $\mathcal F^d_{\beta,\loc}$ contains the potentials considered in \cite{JK,Ste,Saw,SawChan,W} as proper subclasses. 
Previously WUC and SUC properties were derived only for $Y_V=H^{2,\bar{p}}_{\loc}$.
We note that though the dependence of $Y_V$ on $V$ (i.e., $u \in Y_V$ implies $|V|^{\frac{1}{2}}u \in X_2$) does not appear explicitly in the papers cited above, 
it is implicit, see Section \ref{compsect}.

\vspace*{4mm}

Following Carleman, most proofs of unique continuation rely on Carleman type estimates on the norms of the appropriate operators acting from $L^p$ to $L^q$, for certain $p$ and $q$ (e.g., Theorem 2.1 in \cite{JK}, Theorem 1 in \cite{Ste}). Our method is based on an $L^2 \mapsto L^2$ estimate of Proposition \ref{ourlem} and Lemma \ref{sawlem1}, proved in \cite{Saw}. 
In the case $d=3$ we derive Proposition \ref{ourlem} using only Lemma \ref{sawlem1}. The case $d \geq 4$ is reduced to the case $d=3$ at the cost of a more restrictive class of potentials:
the proof uses Stein's interpolation theorem for analytic families of operators \cite{SW}, and relies on Lemma \ref{sawlem2} -- a variant of pointwise inequalities considered in \cite{Saw} and \cite{Ste} (cf. Lemma 1 in \cite{Saw}, Lemma 5 in \cite{Ste}) -- and Lemma \ref{JKlem} of \cite{JK}.  \\

The results of this article have been announced in \cite{KiSh}.

\vspace*{4mm}
\noindent \textbf{Acknowledgments. }We are grateful to Yu.~A.~Semenov for introducing us to the subject of unique continuation and close guidance throughout our work on this article, and to Pierre Milman for his supervision and, in particular, help in communicating our results here.

\vspace*{2mm}

\section{Main Results}

\label{comparisonsect}

Our main results state that (\ref{diffineq}) has the WUC and SUC properties with potentials from $\mathcal F^d_{\beta,\loc}$. The difference between the results is in the classes $Y_V$ within which we look for solutions to (\ref{diffineq}).

\begin{theorem}
\label{mainthm}
There exists a sufficiently small constant $\beta<1$ such that if $V \in \mathcal F^d_{\beta,\loc}$ then (\ref{diffineq}) has the WUC property in  $Y_V^{\weak}$.
\end{theorem}

%The proof of Theorem \ref{mainthm}, given in Section \ref{mainsect}, relies on the key inequality of Proposition \ref{ourlem}.
%Note that while in the case $d \geq 4$ the proof of Proposition \ref{ourlem} employs Stein's interpolation theorem and the inequality of Lemma \ref{sawlem2}, in the case $d=3$ (i.e., when $\mathcal F_{\beta,\loc}^d=F_{\beta,\loc}$) it can be obtained using only Lemma \ref{sawlem2} with $\gamma=0$, proved in \cite{Saw}.
%
%
%Our result on the SUC property is formulated as follows.

\begin{theorem}
\label{mainthm2}
There exists a sufficiently small constant $\beta<1$ such that if $V \in \mathcal F^d_{\beta,\loc}$, then (\ref{diffineq})  has the SUC property in $Y_V^{\strong}$.
\end{theorem}

The proofs of Theorems \ref{mainthm} and \ref{mainthm2} are given in Section \ref{mainsect}. 
%As a corollary, we have the following result.
Concerning the eigenvalue problem, we have the following result.

\begin{theorem}
\label{eigthm}
Suppose that $H$ is defined by (\ref{formsum}) in assumption that (\ref{formineq}) holds. Let us also assume that $V \in F^d_{\beta,\loc}$ for $\beta<1$ sufficiently small, and $\supp(V)$ is compact in $\mathbb R^d$. Then the only solution to the eigenvalue problem
\begin{equation}
\label{eigprob}
Hu=\lambda u, \quad u \in D(H), \quad \lambda>0
\end{equation}
is zero. 
\end{theorem}

\begin{proof}
The following inclusions are immediate from the definition of operator $H$: $$D(H) \subset H^{1,2}(\mathbb R^d) \cap D(V_+^{\frac{1}{2}}) \cap D(V_-^{\frac{1}{2}}),$$ $$D(H) \subset D(H_{\max}),$$
where $$D(H_{\max}):=\{f \in \mathcal H: ~ \Delta f \in \mathcal D'(\mathbb R^d) \cap L^1_{\loc}(\mathbb R^d), Vf \in L^1_{\loc}(\mathbb R^d), -\Delta f+Vf \in \mathcal H\}.$$
Therefore, $D(H) \subset Y_V^{\weak}$ and if $u \in D(H)$ is a solution to (\ref{eigprob}), then $$|\Delta u|=|(V-\lambda) u| \quad \text{ a.e. in }\mathbb R^d.$$ 
By Kato's theorem \cite{Kato1} $u$ has compact support.
Now Theorem \ref{eigthm} follows from Theorem \ref{mainthm}.
\end{proof}

\section{Historical context}

\label{compsect}

%Below we compare our results with classical results that followed Carleman's paper \cite{Carl}. In particular, we identify several subclasses of $\mathcal F^d_{\beta,\loc}$.

\vspace*{2mm}

1) D.~Jerison and C.~Keing \cite{JK} and E.M.~Stein \cite{Ste} proved the validity of the SUC property for potentials from classes $L^{\frac{d}{2}}_{\loc}(\Omega)$ and $L^{\frac{d}{2},\infty}_{\loc}(\Omega)$ (weak type $d/2$ Lorentz space), respectively. Below $\|\cdot\|_{p,\infty}$ denotes weak type $p$ Lorentz norm.
One has
%\begin{proposition}
%\label{stein}
\begin{equation}
\label{incll}
L^{\frac{d}{2}}_{\loc}(\Omega) \subsetneq \bigcap_{\beta>0}\mathcal F^d_{\beta,\loc},
\end{equation}
\begin{equation}
\label{incllweak}
L^{\frac{d}{2},\infty}_{\loc}(\Omega) \subsetneq \bigcup_{\beta>0}\mathcal F^d_{\beta,\loc}.
\end{equation}
The first inclusion follows straightforwardly from the Sobolev embedding theorem. For the following proof of the second inclusion  let us note first that
\begin{equation*}
\|\mathbf{1}_{B(x_0,\rho)} |W|^{\frac{d-1}{4}}(-\Delta)^{-\frac{d-1}{2}}|W|^{\frac{d-1}{4}}\mathbf{1}_{B(x_0,\rho)}\|_{2 \mapsto 2}= \|\mathbf{1}_{B(x_0,\rho)}|V|^{\frac{d-1}{4}}(-\Delta)^{-\frac{d-1}{4}}\|_{2 \mapsto 2}^2.
\end{equation*}
Next, if $V \in L^{d/2,\infty}$, then
\begin{equation}
\label{strichineq}
\|\mathbf{1}_{B(x_0,\rho)}|V|^{\frac{d-1}{4}}(-\Delta)^{-\frac{d-1}{4}}\|_{2 \mapsto 2} \leq \left(\frac{2d^{-1}\pi^{\frac{d}{2}}c_{\frac{1}{2}}}{\Gamma\left(\frac{d}{2} \right)c_{\frac{d}{2}}}\right) \|\mathbf{1}_{B(x_0,\rho)} V\|_{\frac{d}{2},\infty}^{\frac{d-1}{4}},
\end{equation}
which is a special case of Strichartz inequality with sharp constants, proved in \cite{KPS}. Required inclusion follows.

To see that the latter inclusion is strict we introduce a family of potentials
\begin{equation}
\label{steinV}
V(x):=\frac{C\bigl(\mathbf{1}_{B(1+\delta)}(x)-\mathbf{1}_{B(1-\delta)}(x)\bigr)}{\bigl(|x|-1\bigr)^{\frac{2}{d-1}}\left(-\ln \bigl||x|-1\bigr| \right)^b}, \quad \text{ where }b>\frac{2}{d-1}, \quad 0<\delta<1.
\end{equation} 
A straightforward computation shows that 
$V \in \mathcal F^d_{\beta,\loc}$, as well as
$V \in L_{\loc}^{\frac{d-1}{2}}(\Omega) \setminus L_{\loc}^{\frac{d-1}{2}+\varepsilon}(\Omega)$ for any $\varepsilon>0$, so that $V \not \in L_{\loc}^{\frac{d}{2},\infty}(\Omega)$.

\vspace*{2mm}

The result in \cite{Ste} can be formulated as follows.
Suppose that $d \geq 3$ and $V \in L^{\frac{d}{2},\infty}_{\loc}(\Omega)$. There exists a sufficiently small constant $\beta$ such that if 
\begin{equation*}
\sup_{x_0 \in \Omega}\overline{\lim\limits_{\rho \to 0}}\|\mathbf{1}_{B_K(x_0,\rho)}V\|_{\frac{d}{2},\infty} \leq \beta, 
\end{equation*}
then (\ref{diffineq}) has the SUC property in 
 $Y_V:=H^{2,\bar{p}}_{\loc}(\Omega)$, where $\bar{p}:=\frac{2d}{d+2}$. (It is known that the assumption of $\beta$ being sufficiently small can not be omitted, see
 \cite{KT}.) 
 
In view of (\ref{incll}), (\ref{incllweak}), the results in \cite{Ste} and in \cite{JK} follow from Theorem \ref{mainthm2} provided that we show $|V|^{\frac{1}{2}}u \in X_2$. Indeed, let $L^{q,p}$ be the $(q,p)$ Lorentz space (see \cite{SW}). By Sobolev embedding theorem for Lorentz spaces $H_{\loc}^{2,\bar{p}}(\Omega) \hookrightarrow L_{\loc}^{\bar{q},\bar{p}}(\Omega)$ with $\bar{q}:=\frac{2d}{d-2}$ \cite{SW}. Hence, by H\"{o}lder inequality in Lorentz spaces $|V|^{\frac{1}{2}}u \in X_2$ whenever $u \in L_{\loc}^{\bar{q},\bar{p}}(\Omega)$ and $V \in L_{\loc}^{d/2,\infty}$. Also, $H_{\loc}^{2,\bar{p}}(\Omega) \hookrightarrow H_{\loc}^{1,\bar{p}}(\Omega)$, so  $H_{\loc}^{2,\bar{p}}(\Omega) \subset Y^{\strong}_V,$
as required.

\vspace*{2mm}

2) E.T.~Sawyer \cite{Saw} proved uniqueness of continuation for the case $d=3$ and potential $V$ from the local Kato-class
\begin{equation*}
\mathcal K_{\beta,\loc}:=\{W \in L^{1}_{\loc}(\Omega): \sup_K\overline{\lim\limits_{\rho \to 0}} \sup_{x_0 \in K} \|(-\Delta)^{-1} \mathbf{1}_{B_K(x_0,\rho)} |W|\|_{\infty} \leq \beta\},
\end{equation*}
where $K$ is a compact subset of $\Omega$.
It is easy to see that
\begin{equation*}
\mathcal K_{\beta,\loc} \subsetneq F_{\beta,\loc}.
\end{equation*}
To see that the latter inclusion is strict consider, for instance, 
potential
\begin{equation*}
V_{\beta}(x):=\beta v_0, \quad v_0:=\left(\frac{d-2}{2}\right)^2|x|^{-2}.
\end{equation*}
By Hardy's inequality, $V_\beta \in F_{\beta,\loc}$. At the same time, $\|(-\Delta)^{-1} v_0\mathbf{1}_{B(\rho)}\|_{\infty}=\infty$ for all $\rho>0$, hence $V_{\beta} \not \in \mathcal K_{\beta,\loc}$ for all $\beta \ne 0$.

The next statement is essentially due to E.T.~Sawyer \cite{Saw}. 

\begin{theorem}
\label{sawthm}
Let $d=3$. There exists a constant $\beta<1$ such that if $V \in \mathcal K_{\beta,\loc}$ then (\ref{diffineq}) has the WUC property in $Y_V^{\mathcal K}:=\{f \in X_1: \Delta f \in X_1,~Vf \in X_1\}$.
\end{theorem}

The proof of Theorem \ref{sawthm} is provided in Section \ref{exsect}.

Despite the embedding $\mathcal K_{\beta,\loc} \hookrightarrow F_{\beta,\loc}$, Theorem \ref{mainthm} does not imply Theorem \ref{sawthm}. The reason is simple:
$Y_V^{\mathcal K} \not\subset Y_V^{\weak}$.

\vspace*{2mm}

3) S.~Chanillo and E.T.~Sawyer showed in \cite{SawChan} the validity of the SUC property for (\ref{diffineq}) in  $Y_V=H^{2,2}_{\loc}(\Omega)$ ($d \geq 3$) for potentials $V$ locally small in Campanato-Morrey class $M^p$ ($p>\frac{d-1}{2}$), 
\begin{equation*}
M^p:=\{W \in L^p: \|W\|_{M^p}:=\sup_{x \in \Omega,~r>0}r^{2-\frac{d}{p}}\|\mathbf{1}_{B(x,r)} W\|_p<\infty\}.
\end{equation*}

Note that for $p>\frac{d-1}{2}$
\begin{equation*}
M^p_{\loc} \subsetneq \bigcup_{\beta>0} \mathcal F_{\beta,\loc}^d
\end{equation*}
(see \cite{SawChan,F,KS}). 
To see that the above inclusion is strict one may consider, for instance, potential defined in (\ref{steinV}).

It is easy to see, using H\"{o}lder inequality, that if $u \in H^{2,2}_{\loc}(\Omega)$ and $V \in M^p_{\loc}$ ($p>\frac{d-1}{2}$), then $|V|^{\frac{1}{2}}u \in X_2$, i.e., $u \in Y_V^{\weak}$. However, the assumption `$u \in H^{2,2}_{\loc}$' is in general too restrictive for application of this result to the problem of absence of positive eigenvalues (see Remark \ref{proprem}).

\begin{remark}
%[$H^{2,q}$-properties of eigenfunctions of the self-adjoint Schr\"{o}dinger operator $H=(-\Delta \dotplus V_+)\dotplus (-V_-)$]
\label{proprem}
Below we make several comments about $H^{2,q}$-properties of the eigenfunctions of the self-adjoint Schr\"{o}dinger operator $H=(-\Delta \dotplus V_+)\dotplus (-V_-)$, $V=V_+-V_-$, defined by (\ref{formsum}) in the assumption that condition
%We can only show that the eigenfunctions of the self-adjoint Schr\"{o}dinger operator $H=(-\Delta \dotplus V_+)\dotplus (-V_-)$ defined by (\ref{formsum}) for potentials $V=V_+-V_-$ satisfying (\ref{formineq}) with some $\beta<1$ belong to $H^{2,q}$ with a certain $q=q(\beta)$ (see below). Fortunately though $Y_V^{\weak} \supset D(H)$, therefore our result on the WUC property holds for the latter eigenfunctions and the arguments of Kato apply (see Section \ref{comparisonsect}).
%Assume that condition
\begin{equation}
\label{formineq3}
V_- \leq \beta (H_0 \dotplus V_+)+c_{\beta}, \quad \beta<1, c_\beta<\infty
\end{equation}
is satisfied. (Note that (\ref{formineq3}) implies condition (\ref{formineq}). We say that (\ref{formineq3}) is satisfied with $\beta=0$ if (\ref{formineq3}) holds for any $\beta>0$ arbitrarily close to 0, for an appropriate $c_{\beta}<\infty$.)

Let $u \in D(H)$ and $Hu=\mu u$. Then
\begin{equation*}
e^{-tH}u=e^{-t\mu}u, \quad t>0.
\end{equation*}
As is shown in \cite{LS},
for every $2 \leq q< \frac{2d}{d-2}\frac{1}{1-\sqrt{1-\beta}}$ there exists a constant $c=c(q,\beta)>0$ such that
\begin{equation}
\label{LSineq}
\|e^{-tH}f\|_q \leq ct^{-\frac{d}{2}\left(\frac{1}{2}-\frac{1}{q} \right)}\|f\|_2,
\end{equation}
where $f \in L^2=L^2(\mathbb R^d)$. Let us now consider several possible $L^p$ and $L^{p,\infty}$ (as well as $L^p_{\loc}$ and $L^{p,\infty}_{\loc}$) conditions on potential $V$. In each case, the corresponding result on $H^{2,q}$-properties of the eigenfunction $u$  immediately implies the inclusion $|V|^{\frac{1}{2}}u \in L^2$ (respectively, $|V|^{\frac{1}{2}}u \in X_2$) (cf. $D(H)$ and $Y_V^{\weak}$).

%Let us now consider the following cases:

%$H^{2,q}$ properties of eigenfunction $u$ depend on $L^p$ properties of  $V$. 

\vspace*{2mm}

(A) Suppose in addition to (\ref{formineq3}) that $V \in L^{\frac{d-1}{2}}_{\loc}$. Then by H\"{o}lder inequality and (\ref{LSineq}) $Vu \in L^q_{\loc}$ and, due to inclusion $D(H) \subset D(H_{\max})$,  $\Delta u \in L^q_{\loc}$ for any $q$ such that
\begin{equation*}
\frac{1}{q}>\frac{2}{d-1}+\frac{d-2}{d}\frac{1-\sqrt{1-\beta}}{2}.
\end{equation*}
The latter implies that $q<2$ in general, i.e., when $\beta$ in (\ref{formineq3}) is close to $1$. Hence, in general the assumption `$u \in H_{\loc}^{2,2}$' is too restrictive for applications to the problem of absence of positive eigenvalues even under additional hypothesis
of the type $V \in L^p_{\comp}$, $\frac{d-1}{2}<p<\frac{d}{2}$
 or $V \in M^p_{\comp}$, $\frac{d-1}{2}<p<\frac{d}{2}$ (cf. \cite{SawChan,RV}).

\vspace*{2mm}

(B1) If $V=V_1+V_2 \in L^p+L^\infty$, $p>\frac{d}{2}$, then (\ref{formineq3}) holds with $\beta=0$ and $u \in L^\infty$. Moreover, it follows that $u \in C^{0,\alpha}$ for any $\alpha \in (0,1-\frac{2}{d}]$. Therefore, $u \in H_{\loc}^{2,p}$ and, in particular, for $d \geq 4$, $u \in H^{2,2}$. 

%By making use of H\"{o}lder inequality one can easily show that in these assumptions of $V$ $$|V_1|^{\frac{1}{2}}u \in L^2$$ (cf. $D(H)$).

\vspace*{2mm}

(B2) Assume in addition to (\ref{formineq3}) that $V \in L^p_{\loc}$, $p>\frac{d}{2}$, and $\beta=0$. Then $u \in H^{2,\underline{p}}_{\loc}$, $\underline{p}>\frac{d}{2}$. If $d=3$, and $p>\frac{d}{2}$ is close to $\frac{d}{2}$, then $u \not\in H^{2,2}_{\loc}$, but of course $u \in H^{2,\bar{p}}_{\loc}$, $\bar{p}=\frac{2d}{d+2}$~($<\underline{p}$). 

%In these assumptions on $V$ the condition $|V|^{\frac{1}{2}}u \in X_2$ is trivially satisfied (cf. definition of space $Y_V^{\weak}$).

\vspace*{2mm}

(B3) If $V=V_1+V_2 \in L^{\frac{d}{2}}+L^\infty$, then (\ref{formineq3}) is satisfied with $\beta=0$ and $u \in \cap_{2 \leq r<\infty} L^r$. Therefore $u \in H^{2,q}_{\loc}$, $q<\frac{d}{2}$ (cf. Remark in \cite{ABG}). In particular, $u \in H^{2,\bar{p}}_{\loc}$ (cf. \cite{JK}). But for $d \geq 5$ it follows $u \in H^{2,2}_{\loc}$. 

%Similarly, we have $|V_1|^{\frac{1}{2}}u \in L^2$.

\vspace*{2mm}

(B4) Finally, suppose that $V=V_1+V_2 \in L^{\frac{d}{2},\infty}+L^\infty$ is such that 
$$\beta:=\left(\frac{d^{-1}\pi^{\frac{d}{2}} \Gamma \left(\frac{d}{4}-\frac{1}{2} \right)}{\Gamma \left(\frac{d}{2}\right)\Gamma \left(\frac{d}{4}+\frac{1}{2} \right)}\right) \|V_1\|_{\frac{d}{2},\infty}<1.$$
Then we have
\begin{equation}
\label{formineq4}
|V| \leq \beta H_0+c_{\beta}, \quad c_\beta<\infty
\end{equation}
and, at the same time,
\begin{equation*}
\|V(\lambda+H_{0,\bar{p}})^{-1}\|_{\bar{p} \mapsto \bar{p}} \leq \beta, \quad \lambda \geq \frac{c_\beta}{\beta}
\end{equation*}
(see \cite{KPS}),
where $H_{0,\bar{p}}$ stands for the extension of $-\Delta$ in $L^{\bar{p}}$ with $D(H_{0,\bar{p}})=H^{2,\bar{p}}$. 
The first inequality implies condition (\ref{formineq3}) and, hence, allows us to conclude that the form sum $H:=H_0\dotplus V$ is well defined. In turn, the second inequality implies existence of the algebraic sum $\hat{H}_{\bar{p}}:=H_{0,\bar{p}}+V$ defined in $L^{\bar{p}}$ with $D(\hat{H}_{\bar{p}})=H^{2,\bar{p}}$, which
coincides with $H$ on the intersection of domains $D(H) \cap H^{2,\bar{p}}$. By making use of the representation
\begin{equation*}
(\lambda+\hat{H}_{\bar{p}})^{-1}=(\lambda+H_{0,\bar{p}})(1+V(\lambda+H_{0,\bar{p}})^{-1})^{-1}
\end{equation*}
one immediately obtains that $(\lambda+\hat{H}_{\bar{p}})^{-1}:L^{\bar{p}} \mapsto L^2$, i.e., any eigenfunction of operator $\hat{H}_{\bar{p}}$ belongs to $L^2$. Furthermore, an analogous representation for $(\lambda+H)^{-1}$ yields the identity
\begin{equation*}
(\lambda+H)^{-1}f=(\lambda+\hat{H}_{\bar{p}})^{-1}f, \quad f \in L^2 \cap L^{\bar{p}}.
\end{equation*}
Therefore, any eigenfunction of $\hat{H}_{\bar{p}}$ is an eigenfunction of $H$ (cf. \cite{Ste}). The converse statement is valid, e.g., for eigenfunctions having compact support.

%Finally, the second resolvent identity implies
%\begin{equation*}
%(\lambda+H)^{-1}:L^2 \cap L^{\bar{p}} \mapsto H^{2,\bar{p}}.
%\end{equation*}
%Therefore, any eigenfunction $u$ of $H$ having compact support belongs to $H^{2,\bar{p}}$ (cf. \cite{Ste}). 

%It is easy to see that if $u$ is an eigenfunction of operator $H_{\bar{p}}$ having compact support, then $u$ is an eigenfunction of operator $H$.

%\dots

If $V \in L^{\frac{d}{2},\infty}_{\loc}$ and (\ref{formineq4}) holds, then $u \in H^{2,q_0}_{\loc}$ for some $q_0>\bar{p}$. Indeed, we have $V \in L^r_{\loc}$ for any $r<\frac{d}{2}$, and so by (\ref{LSineq}) $u \in L^p$ for some $p>\frac{2d}{d-2}$. Thus, $Vu \in L^{q_0}_{\loc}$ for a certain $q_0>\bar{p}$ and, hence, $u \in H^{2,q_0}_{\loc}$. The latter confirms that the result in \cite{Ste}) applies to the problem of absence of positive eigenvalues.

% The inclusion $|V|^{\frac{1}{2}}u \in L^2$ is now immediate.

\end{remark}

\section{Proofs of Theorems \ref{mainthm} and \ref{mainthm2}}

\label{mainsect}

Let us introduce some notations. 
In what follows, we omit index $K$ in $B_K(x_0,\rho)$, and write simply $B(x_0,\rho)$.

\vspace*{2mm}

Let $W \in X_{\frac{d-1}{2}}$, $x_0 \in \Omega$, $\rho>0$, $d \geq 3$, define
\begin{equation}
\label{taudef}
\tau(W,x_0,\rho):=\|\mathbf{1}_{B(x_0,\rho)} |W|^{\frac{d-1}{4}}(-\Delta)^{-\frac{d-1}{2}}|W|^{\frac{d-1}{4}}\mathbf{1}_{B(x_0,\rho)}\|_{2 \mapsto 2}.
\end{equation}
%Then potential $W$ belongs to $\mathcal F^d_{\beta,\loc}$ if and only if
%$\overline{\lim\limits_{\rho \to 0}} \sup_{x_0 \in \Omega}\tau(W,x_0,\rho) \leq \beta$.

Let
$\mathbf{1}_{B(\rho \setminus a)}$ be the characteristic function of set $B(0,\rho) \setminus B(0,a)$, where $0<a<\rho$, and
\begin{equation*}
N_{d}^\delta:=N+\left(\frac{d}{2}-\delta\right)\frac{d-3}{d-1}.
\end{equation*}
We define integral operator
\begin{equation*}
\left[(-\Delta)^{-\frac{z}{2}}\right]_N f(x):=\int_{\mathbb R^d} \left[(-\Delta)^{-\frac{z}{2}}\right]_N(x,y) f(y)dy, \quad 0 \leq \Real(z) \leq d-1
\end{equation*}
whose kernel $\bigl[(-\Delta)^{-\frac{z}{2}}\bigr]_N(x,y)$ is defined by subtracting Taylor polynomial of degree $N-1$ at $x=0$ of function $x \mapsto |x-y|^{z-d}$,
\begin{equation*}
\left[(-\Delta)^{-\frac{z}{2}}\right]_N(x,y):=c_z \left( |x-y|^{z-d} -
\sum_{k=0}^{N-1} \frac{(x \cdot \nabla)^k}{k!}|0-y|^{z-d}
%T_{x,0}^{N-1}\bigl(|x-y|^{z-d}\bigr)
\right),
\end{equation*}
where $(x \cdot \nabla)^k:=\sum_{|\alpha|=k} \frac{k!}{\alpha_1!\dots\alpha_d!}x^\alpha \frac{\partial^k}{\partial x_1^{\alpha_1}\dots x_n^{\alpha_n}}$ is the multinomial expansion of $(x \cdot \nabla)$.
%$T_{x,0}^{N-1}(f)$ is the Taylor polynomial of function $f$ with respect to variable $x$ at $0$ of degree $N-1$.
Define, further,
%Further, denote
\begin{equation*}
\left[(-\Delta)^{-\frac{z}{2}}\right]_{N,t}:=\varphi_t\left[(-\Delta)^{-\frac{z}{2}}\right]_N\varphi_t^{-1},
\end{equation*}
where $\varphi_{t}(x):=|x|^{-t}$. 

Note that if $V$ is a potential from our class $\mathcal F^d_{\beta,\loc}$, and $V_1:=|V|+1$, then for a fixed $x_0 \in \Omega$
\begin{equation}
\label{onerel}
\tau(V_1,x_0,\rho) \leq \tau(V,x_0,\rho)+\varepsilon(\rho),
\end{equation}
where $\varepsilon(\rho) \to 0$ as $\rho \to 0$.

%, and
%if a function $u$ satisfies differential inequality (\ref{diffineq}), then $u$ satisfies inequality (\ref{diffineq}) with $V$ replaced by $V_1$, as well. 

\subsection{Proof of Theorem \ref{mainthm}}Our proof is based on inequalities of Proposition \ref{ourlem} and Lemma \ref{sawlem1}.

\begin{proposition}
\label{ourlem}
If $\tau(V,0,\rho)<\infty$, then there exists a constant $C=C(\rho,\delta,d)>0$ such that 
\begin{equation*}
%\label{reqest}
\|\mathbf{1}_{B(\rho \setminus a)}|V|^{\frac{1}{2}}\left[(-\Delta)^{-1}\right]_{N,N_{d}^\delta}|V|^{\frac{1}{2}}\mathbf{1}_{B(\rho \setminus a)}\|_{2 \mapsto 2} \leq C\tau(V,0,\rho)^{\frac{1}{d-1}},
\end{equation*}
where $0<\delta<1/2$, for all positive integers $N$.
\end{proposition}

\begin{lemma}
\label{sawlem1}
There exists a constant $C=C(d)$ such that
\begin{equation*}
%\left|\psi_{2-d}(x-y)-\sum_{k=0}^{N-1} \frac{(x \cdot \nabla)}{k!}\psi_{2-d}(-y) \right| 
\left|\left[(-\Delta)^{-1}\right]_N(x,y) \right|
\leq 
C N^{d-3}\left(\frac{|x|}{|y|} \right)^N (-\Delta)^{-1}(x,y)
\end{equation*}
for all $x$, $y \in \mathbb R^d$ and all positive integers $N$.
\end{lemma}

In turn, the proof of Proposition \ref{ourlem} in the case $d=3$ follows immediately from Lemma \ref{sawlem1} which is a simple consequence of Lemma 1 in \cite{Saw} (i.e., Lemma \ref{sawlem2} below for $\gamma=0$). In the case that $d \geq 4$ we prove Proposition \ref{ourlem} using Stein's interpolation theorem and the estimates of Lemma \ref{JKlem}, which is due to 
D.~Jerison and C.~Kenig \cite{JK}, and Lemma \ref{sawlem2}, which generalizes the inequalities considered in \cite{Saw} and \cite{Ste} (cf. Lemma 1 in \cite{Saw} and Lemma 5 in \cite{Ste}).

\begin{lemma}[\cite{JK}]
\label{JKlem}
There exist constants $C_2=C_2(\rho_1,\rho_2,\delta,d)$ and $c_2=c_2(\rho_1,\rho_2,\delta,d)>0$ such that
\begin{equation*}
\|\mathbf{1}_{B(\rho_1 \setminus a)} \left[(-\Delta)^{-i\gamma}\right]_{N,N+\frac{d}{2}-\delta} \mathbf{1}_{B(\rho_2 \setminus a)}\|_{2 \mapsto 2} \leq C_2 e^{c_2|\gamma|},
\end{equation*}
where $0<\delta<1/2$, for all $\gamma \in \mathbb R$ and all positive integers $N$.
\end{lemma}

\begin{lemma}
\label{sawlem2}
There exist constants $C_1=C_1(d)$ and $c_1=c_1(d)>0$ such that
\begin{equation*}
\left|\left[(-\Delta)^{-\frac{d-1+i\gamma}{2}}\right]_N(x,y) \right| \leq C_1e^{c_1\gamma^2} \left(\frac{|x|}{|y|} \right)^N (-\Delta)^{-\frac{d-1}{2}}(x,y)
\end{equation*}
for all $x$, $y \in \mathbb R^d$, all $\gamma \in \mathbb R$ and all positive integers $N$.
\end{lemma}

We prove Lemma \ref{sawlem2} at the end of this section.

\begin{proof}[Proof of Proposition \ref{ourlem}]
In the case that $d=3$ result follows immediately from Lemma \ref{sawlem1}, proved in \cite{Saw}.
In the case $d \geq 4$ the proof can be obtained, using Lemmas \ref{JKlem} and \ref{sawlem2}, by making use of Stein's interpolation theorem (see, e.g., \cite{SW}). Indeed, consider the operator-valued function
\begin{equation*}
F(z):=\mathbf{1}_{B(\rho \setminus a)} |V|^{\frac{d-1}{4}z}\varphi_{N+\left(\frac{d}{2}-\delta\right)(1-z)} \left[(-\Delta)^{-\frac{d-1}{2}z}\right]_N\varphi_{N+\left(\frac{d}{2}-\delta\right)(1-z)}^{-1}|V|^{\frac{d-1}{4}z}\mathbf{1}_{B(\rho \setminus a)} 
\end{equation*}
defined on the strip $\{z \in \mathbb C: 0 \leq \Real(z) \leq 1\}$. 
By Lemma \ref{JKlem},
\begin{equation*}
\|F(i\gamma)\|_{2 \mapsto 2} \leq C_2e^{c_2|\gamma|}, \quad \gamma \in \mathbb R,
\end{equation*}
and by Lemma \ref{sawlem2} and definition of norm $\tau(V,0,\rho)$ (see (\ref{taudef}))
\begin{equation*}
\|F(1+i\gamma)\|_{2 \mapsto 2} \leq \tau(V,0,\rho) C_1e^{c_1\gamma^2}, \quad \gamma \in \mathbb R.
\end{equation*}
Together with obvious observations about analyticity of $F$ this implies that 
$F$ satisfies all conditions of Stein's interpolation theorem. In particular, $F\bigl(\frac{2}{d-1}\bigr)$ is a bounded $L^2 \mapsto L^2$,
which completes the proof of Proposition \ref{ourlem}.
\end{proof}

\begin{proof}[Proof of Theorem \ref{mainthm}]
Let $u \in Y_V^{\weak}$. Without loss of generality we may assume $u \equiv 0$ on $B(0,a)$ for $a>0$ sufficiently small, such that there exists $\rho>a$ with the properties $\rho<1$ and $\bar{B}(0,3\rho) \subset \Omega$.
In order to prove that $u$ vanishes on $\Omega$ it suffices to show that $u \equiv 0$ on $B(0,\rho)$ for any such $\rho$.

Let $\eta \in C_0^\infty(\Omega)$ be such that $0 \leq \eta \leq 1$, $\eta \equiv 1$ on $B(0,2\rho)$, $\eta \equiv 0$ on $\Omega \setminus B(0,3\rho)$, $|\nabla \eta| \leq \frac{c}{\rho}$, $|\Delta \eta| \leq \frac{c}{\rho^2}$.
Let $E_\eta(u):=2 \nabla \eta \nabla u+u\Delta \eta \in X_1$.  Denote $u_\eta:=u\eta$. Since $\mathcal L^{2,1}_{\loc}(\Omega) \subset H_{\loc}^{1,p}(\Omega)$, $p<\frac{d}{d-1}$, we have $E_\eta(u) \in L_{\comp}^1(\Omega)$
and hence
\begin{equation*}
\Delta u_\eta=\eta\Delta u+E_\eta(u)
\end{equation*}
implies
$\Delta u_\eta \in L_{\comp}^1(\Omega)$. Thus, we can write
\begin{equation*}
u_\eta=(-\Delta)^{-1}(-\Delta u_\eta).
\end{equation*}
The standard limiting argument (involving consideration of $C^\infty_0$-mollifiers, subtraction of Taylor polynomial of degree $N - 1$ at
$0$ of function $u_\eta$ and interchanging the signs of differentiation and integration) allows us to conclude further
\begin{equation}
\label{identity}
u_\eta=[(-\Delta)^{-1}]_N(-\Delta u_\eta).
\end{equation}
Let us denote $\mathbf{1}_{B(\rho)}^c:=1-\mathbf{1}_{B(\rho)}$, so that $\Delta u_\eta=(\mathbf{1}_{B(\rho \setminus a)}+\mathbf{1}_{B(\rho)}^c) \Delta u_\eta$. Observe that 
\begin{equation*}
\supp~\eta\Delta u\subset \bar{B}(0,3\rho) \setminus B(0,a), \quad \supp~ E_\eta(u)  \subset \bar{B}(0,3\rho) \setminus B(0,2\rho)
\end{equation*}
and, thus, $\mathbf{1}_{B(\rho)}^c \eta\Delta u=\mathbf{1}_{B(3\rho \setminus \rho)} \Delta u$, $\mathbf{1}_{B(\rho)}^c E_\eta(u)=\mathbf{1}_{B(3\rho \setminus 2\rho)} E_\eta(u)$.
Identity (\ref{identity}) implies then
\begin{multline}
\notag
%\label{id6}
\mathbf{1}_{B(\rho)} V_1^{\frac{1}{2}}\varphi_{N_d^{\delta}} u=\mathbf{1}_{B(\rho)} V_1^{\frac{1}{2}}[(-\Delta)^{-1}]_{N,N_{d}^\delta}V_1^{\frac{1}{2}}\mathbf{1}_{B(\rho \setminus a)}\varphi_{N_{d}^\delta}\frac{-\Delta u}{V_1^{\frac{1}{2}}}+\\+\mathbf{1}_{B(\rho)} V_1^{\frac{1}{2}}[(-\Delta)^{-1}]_{N,N_{d}^\delta} V_1^{\frac{1}{2}}\mathbf{1}^c_{B(\rho)}\varphi_{N_{d}^\delta}\frac{-\eta \Delta u }{V_1^{\frac{1}{2}}}+\\+
\mathbf{1}_{B(\rho)} V_1^{\frac{1}{2}}[(-\Delta)^{-1}]_{N,N_{d}^\delta} \mathbf{1}_{B(3\rho \setminus 2\rho)}\varphi_{N_{d}^\delta}(-E_\eta(u))
\end{multline}
(we assume that $0<\delta<1/2$ is fixed throughout the proof)
or,  letting $I$ to denote the left hand side and, respectively, $I_1$ ,
$I_1^c$ and $I_2$  the three summands of the right hand side of the last
equality, we rewrite the latter as
\begin{equation*}
I=I_1+I_1^c+I_2.
\end{equation*}
We would like to emphasize that a priori $I \not\in L^2$, but only $I \in L^s$, $s<d/(d-2)$. Hence, in the case that $d \geq 4$ we must first prove that $I_1$, $I_1^c$ and $I_2$ are in $L^2$, so that $I \in L^2$ as well.
Therefore, we obtain estimates $\|I_1^c\|_2 \leq c_1 \varphi_{N_d^\delta}(\rho)$, $\|I_2\|_2 \leq c_2 \varphi_{N_d^\delta}(\rho)$ and $\|I_1\|_2 \leq \alpha \|I\|_2$, $\alpha<1$, and conclude that
%\begin{equation*}
$(1-\alpha) \|I\|_2 \leq (c_1+c_2) \varphi_{N_d^\delta}(\rho)$, and therefore that
%\end{equation*}
\begin{equation*}
\left\|\mathbf{1}_{B(\rho \setminus a)} \frac{\varphi_{N_d^\delta}}{\varphi_{N_d^\delta}(\rho)}u\right\|_2 \leq \frac{c_1+c_2}{1-\alpha}.
\end{equation*}
Letting $N \to \infty$, we derive identity $u \equiv 0$ in $B(0,\rho)$.

1) %Let us look at $I_1$. 
\textit{Proof of $I_1 \in L^2$ and $\|I_1\|_2 \leq \alpha \|I\|_2$, $\alpha<1$.}
Observe that $$\mathbf{1}_{B(\rho \setminus a)}\frac{|\Delta u|}{V_1^{1/2}}  \leq \mathbf{1}_{B(\rho)}\frac{|V||u|}{V_1^{1/2}}  \leq \mathbf{1}_{B(\rho)}|V|^{1/2}|u| \in X_2 \quad (\text{since } u \in Y_V^{\weak}),$$ 
and hence, according to Proposition \ref{ourlem},
\begin{equation*}
%\notag
\|I_1\|_{2} \leq \left\|\mathbf{1}_{B(\rho \setminus a)} V_1^{\frac{1}{2}}[(-\Delta)^{-1}]_{N,N_{d}^\delta}V_1^{\frac{1}{2}}\mathbf{1}_{B(\rho \setminus a)}\right\|_{2 \mapsto 2}\left\|\mathbf{1}_{B(\rho)}\varphi_{N_{d}^\delta}|V|^{\frac{1}{2}} u\right\|_{2} \leq \beta_1 \|\mathbf{1}_{B(\rho)} \varphi_{N_{d}^\delta}|V|^{\frac{1}{2}} u\|_{2}.
%\leq 
%\\ \leq C(\beta_1+\delta)^{\frac{2}{d-1}}\|\mathbf{1}_{\rho \setminus a} \varphi_{N_{d}^\delta}|V|^{\frac{1}{2}} u\|_{2}.
\end{equation*}
Here
%\begin{equation*}
$\beta_1:=C \tau(V_1,0,\rho)^{\frac{1}{d-1}}$,
%\end{equation*}
where $C$ is the constant in formulation of Proposition \ref{ourlem}. We may assume that $\beta_1<1$ (see (\ref{onerel})).

2) \textit{Proof of $\|I_1^c\|_2 \leq c_1 \varphi_{N_d^\delta}(\rho)$.} By Proposition \ref{ourlem},
\begin{multline}
\notag
\|I_1^c\|_{2} \leq \left\|\mathbf{1}_{B(\rho \setminus a)} V_1^{\frac{1}{2}}[(-\Delta)^{-1}]_{N,N_{d}^\delta}V_1^{\frac{1}{2}}\mathbf{1}_{B(3\rho \setminus \rho)}\right\|_{2 \mapsto 2}\left\|\mathbf{1}_{B(\rho)}^c\varphi_{N_{d}^\delta}|V|^{\frac{1}{2}} u\right\|_{2} \leq \\ \leq \beta_2 \varphi_{N_{d}^\delta}(\rho)\|\mathbf{1}_{B(3\rho)} |V|^{1/2}u\|,
%\leq 
%\\ \leq C(\beta_1+\delta)^{\frac{2}{d-1}}\|\mathbf{1}_{\rho \setminus a} \varphi_{N_{d}^\delta}|V|^{\frac{1}{2}} u\|_{2}.
\end{multline}
where
%\begin{equation*}
$\beta_2:=C \tau(V_1,0,3\rho)^{\frac{1}{d-1}}<\infty$. \\ [-3mm]
%\end{equation*}

3) \textit{Proof of $\|I_2\|_2 \leq c_2 \varphi_{N_d^\delta}(\rho)$.} We need to derive an estimate of the form
\begin{equation*}
\|I_2\|_2 \leq C \varphi_{N_{d}^\delta}(\rho) \|E_\eta(u)\|_1,
\end{equation*}
where $C$ can depend on $d$, $\delta$, $a$, $\rho$, $\|\mathbf{1}_{B(\rho)} V\|_1$, but not on $N$. We have
\begin{multline}
\notag
\|I_2\|_2 \leq \left\|\mathbf{1}_{B(\rho \setminus a)} V_1^{1/2} [(-\Delta)^{-1}]_{N,N_{d}^\delta} \mathbf{1}_{B(3\rho \setminus 2\rho)}\right\|_{1 \mapsto 2}\left\|\mathbf{1}_{B(3\rho \setminus 2\rho)} \varphi_{N_{d}^\delta} E_\eta(u)\right\|_1 \leq \\ \left\|\mathbf{1}_{B(\rho \setminus a)} V_1^{1/2} [(-\Delta)^{-1}]_{N,N_{d}^\delta} \mathbf{1}_{B(3\rho \setminus 2\rho)}\right\|_{1 \mapsto 2} 2^{-N}\varphi_{N_{d}^\delta}(\rho) \left\|E_\eta(u)\right\|_1.
\end{multline}
Now for $h \in L^1(\mathbb R^d)$, in virtue of Lemma \ref{sawlem1},
\begin{multline}
\notag
\|\mathbf{1}_{B(\rho \setminus a)} V_1^{1/2} [(-\Delta)^{-1}]_{N,N_{d}^\delta} \mathbf{1}_{B(3\rho \setminus 2\rho)} h\|_2 \leq \\ \leq \|\mathbf{1}_{B(\rho)} V_1^{1/2}\|_2 \|\mathbf{1}_{B(\rho \setminus a)} [(-\Delta)^{-1}]_{N,N_d^\delta} \mathbf{1}_{B(3\rho \setminus 2\rho)} h\|_{\infty} \leq \\
\leq \|\mathbf{1}_{B(\rho)} V_1^{1/2}\|_2 CN^{d-3} \varphi_{\left(\frac{d}{2}-\delta\right)\frac{d-3}{d-1}}(a)\varphi_{\left(\frac{d}{2}-\delta\right)\frac{d-3}{d-1}}^{-1}(3\rho) \|\mathbf{1}_{B(\rho)} (-\Delta)^{-1} \mathbf{1}_{B(3\rho \setminus 2\rho)}h\|_\infty \leq \\ \leq 
(\|\mathbf{1}_{B(\rho)}\|_1+\|\mathbf{1}_{B(\rho)} V\|_1)^{1/2} CN^{d-3} \left(\frac{3\rho}{a} \right)^{\left(\frac{d}{2}-\delta\right)\frac{d-3}{d-1}} M_{\rho},
\end{multline}
where 
\begin{equation*}
M_{\rho}:= C_2 \esssup_{x \in B(0,\rho)} \int_{2\rho \leq |y| \leq 3\rho} |x-y|^{2-d}|h(y)|dy \leq C_2 \rho^{2-d}\|h\|_1.
\end{equation*}
Therefore
\begin{multline}
\notag
\left\|\mathbf{1}_{B(\rho \setminus a)} V_1^{1/2} [(-\Delta)^{-1}]_{N,N_{d}^\delta} \mathbf{1}_{B(3\rho \setminus 2\rho)}\right\|_{1 \mapsto 2} \leq \\ \leq (\|\mathbf{1}_B(\rho)\|_1+\|\mathbf{1}_{B(\rho)} V\|_1)^{1/2} CC_2 N^{d-3} \left(\frac{3\rho}{a} \right)^{\left(\frac{d}{2}-\delta\right)\frac{d-3}{d-1}}\rho^{2-d}
\end{multline}
Hence, there exists a constant $\hat{C}=\hat{C}(d,\delta,a,\rho,\|\mathbf{1}_{B(\rho)} V\|_1)$ such that
\begin{equation*}
\|I_2\|_2 \leq \hat{C}N^{d-3} 2^{-N} \varphi_{N_d^\delta}(\rho)\|E_\eta(u)\|_1,
\end{equation*}
which implies the required estimate.
\end{proof}

\begin{proof}[Proof of Lemma \ref{sawlem2}]
%The proof of inequality (\ref{sawineq}) can be derived from the proof of its special case for $\gamma=0$ (i.e., from the proof of Lemma 1 in \cite{Saw}), once we note the following. Put
%\begin{equation*}
%\left[\begin{array}{c} -\frac{1}{2}+\frac{i\gamma}{2} \\ k \end{array} \right]:=\prod_{j=1}^k \left(1+\frac{-\frac{1}{2}+\frac{i\gamma}{2}}{j} \right).
%\end{equation*}
%Then
%\begin{multline}
%\notag
%\left|~\left[\begin{array}{c} -\frac{1}{2}+\frac{i\gamma}{2} \\ k \end{array} \right]~\right|=\prod_{j=1}^k \left(1-\frac{1}{2j} \right) \prod_{j=1}^k \sqrt{1+\frac{\gamma^2}{(2j-1)^2}} \leq \\ \leq \prod_{j=1}^k \left(1-\frac{1}{2j} \right) e^{\gamma^2\sum_{j=1}^k \frac{1}{(2j-1)^2}} \leq \prod_{j=1}^k \left(1-\frac{1}{2j} \right) e^{c\gamma^2}, \quad c=\frac{\pi^2}{48}.
%\end{multline}
%We modify the proof of Lemma 1 in \cite{Saw}. 
The proof essentially follows the argument in \cite{Saw}.
Put
\begin{equation*}
\left[\begin{array}{c} -\frac{1}{2}+\frac{i\gamma}{2} \\ k \end{array} \right]:=\prod_{j=1}^k \left(1+\frac{-\frac{1}{2}+\frac{i\gamma}{2}}{j} \right).
\end{equation*}
Then
\begin{multline}
\label{aest111}
\left|~\left[\begin{array}{c} -\frac{1}{2}+\frac{i\gamma}{2} \\ k \end{array} \right]~\right|=\prod_{j=1}^k \left(1-\frac{1}{2j} \right) \prod_{j=1}^k \sqrt{1+\frac{\gamma^2}{(2j-1)^2}} \leq \\ \leq \prod_{j=1}^k \left(1-\frac{1}{2j} \right) e^{\gamma^2\sum_{j=1}^k \frac{1}{(2j-1)^2}} \leq \prod_{j=1}^k \left(1-\frac{1}{2j} \right) e^{\gamma^2 c}, \quad c=\frac{\pi^2}{48}.
\end{multline}
%(This estimate extends the corresponding result in \cite{Saw} to non-zero values of $\gamma$.) 
We may assume, after a dilation and rotation, that $x=(x_1,x_2,0,\dots,0)$, $y=(1,0,\dots,0)$. Thus, passing to polar coordinates $(x_1,x_2)=te^{i\theta}$, we reduce our inequality to inequality
\begin{equation*}
\left||1-te^{i\theta}|^{-1-i\gamma}-P_{N-1}(t,\theta) \right| \leq C e^{c\gamma^2} t^{N}|1-te^{i\theta}|^{-1}, \quad \text{ for all } \gamma \in \mathbb R
\end{equation*}
and for appropriate $C>0$, $c>0$. Here $P_{N-1}(t,\theta)$ denotes the Taylor polynomial of degree $N-1$ at point $z=0$ of function $z=te^{i\theta} \mapsto |1-z|^{-1}$. Similarly to \cite{Saw}, via summation of geometric series we obtain a representation
\begin{equation*}
P_{N-1}(t,\theta)=\sum_{m=0}^{N-1} a^{\gamma}_m(\theta)t^m,
\end{equation*}
where 
\begin{equation*}
a^{\gamma}_m(\theta):=\sum_{k+l=m} \left[\begin{array}{c} -\frac{1}{2}+\frac{i\gamma}{2} \\ l \end{array} \right]~\left[\begin{array}{c} -\frac{1}{2}+\frac{i\gamma}{2} \\ k \end{array} \right] e^{i(k-l)\theta}.
\end{equation*}
Note that
\begin{equation*}
a_m^{0}(0)=\sum_{k+l=m}\left[\begin{array}{c} -\frac{1}{2} \\ l \end{array} \right]~\left[\begin{array}{c} -\frac{1}{2} \\ k \end{array} \right]=1
\end{equation*}
since
\begin{equation*}
\sum_{m=0}^\infty a_m^{0}(0)t^m=(1-t)^{-1}=\sum_{m=0}^\infty t^m.
\end{equation*}
Now estimate (\ref{aest111}) and identity $a_m^{0}(0)=1$ yield
\begin{equation*}
|a^{\gamma}_m(\theta)| \leq \sum_{k+l=m}\left|~\left[\begin{array}{c} -\frac{1}{2} \\ l \end{array} \right]~\right|~\left|~\left[\begin{array}{c} -\frac{1}{2} \\ k \end{array} \right]~\right|e^{2c\gamma^2}=e^{2c\gamma^2}.
\end{equation*}
We have to distinguish between four cases $t \geq 2$, $1<t<2$, $0\leq t \leq \frac{1}{2}$ and $\frac{1}{2} <t<1$. Below we consider only the cases $t \geq 2$ and $1<t<2$ (proofs in two other cases are similar).

If $t \geq 2$, then
\begin{equation*}
|P_{N-1}(t,\theta)| \leq \sum_{m=0}^{N-1} |a^{\gamma}_m(\theta)|t^m \leq e^{2c\gamma^2}t^N \leq \frac{3}{2} e^{2c\gamma^2}t^N|1-te^{i\theta}|^{-1}
\end{equation*}
since $1 \leq \frac{3}{2}t|1-te^{i\theta}|^{-1}$. Hence, using $\bigl|~|1-te^{i\theta}|^{-1-i\gamma}~\bigr| \leq t^N|1-te^{i\theta}|^{-1}$, it follows
\begin{equation*}
\left||1-te^{i\theta}|^{-1-i\gamma}-P_{N-1}(t,\theta) \right| \leq t^N|1-te^{i\theta}|^{-1}+\frac{3}{2} e^{2c\gamma^2}t^N|1-te^{i\theta}|^{-1} \leq C e^{2c\gamma^2}t^N|1-te^{i\theta}|^{-1}
\end{equation*}
for an appropriate $C>0$, as required.

If $1<t<2$, then, after two summations by parts, we derive
\begin{multline}
\notag
P_{N-1}(t,\theta)=\sum_{l=0}^{N-3} S \left[\begin{array}{c} -\frac{1}{2}+\frac{i\gamma}{2} \\ l \end{array} \right] D_l(\bar{z})\sum_{k=0}^{N-l-3}S\left[\begin{array}{c} -\frac{1}{2}+\frac{i\gamma}{2} \\ k \end{array} \right]D_k(z)+\\
+\sum_{l=0}^{N-2} S \left[\begin{array}{c} -\frac{1}{2}+\frac{i\gamma}{2} \\ l \end{array} \right] \left[\begin{array}{c} -\frac{1}{2}+\frac{i\gamma}{2} \\ N-l-2 \end{array} \right]D_l(\bar{z})D_{N-l-2}(z)+\\
+\sum_{k=0}^{N-1}\left[\begin{array}{c} -\frac{1}{2}+\frac{i\gamma}{2} \\ k \end{array} \right]~\left[\begin{array}{c} -\frac{1}{2}+\frac{i\gamma}{2} \\ N-k-1 \end{array} \right]z^kD_{N-1-k}(z)=J_1+J_2+J_3,
\end{multline}
where
\begin{equation*}
S\left[\begin{array}{c} \delta \\ k \end{array} \right]:=\left[\begin{array}{c} \delta \\ k \end{array} \right]-\left[\begin{array}{c} \delta \\ k+1 \end{array} \right], \quad D_k(z):=\sum_{j=0}^k z^j.
\end{equation*}
We use estimate
\begin{equation*}
\left|~S\left[\begin{array}{c} -\frac{1}{2}+\frac{i\gamma}{2} \\ k \end{array} \right]~\right|=\left|~ \left[\begin{array}{c} -\frac{1}{2}+\frac{i\gamma}{2} \\ k \end{array} \right] \left(\frac{-\frac{1}{2}+\frac{i\gamma}{2}}{1+k} \right)~\right| \leq C(k+1)^{-\frac{1}{2}}e^{c\gamma^2}
\end{equation*}
to obtain, following an argument in \cite{Saw}, that each $J_i$ ($i=1,2,3$) is majorized by $Ce^{c\gamma^2}t^N|1-te^{i\theta}|^{-1}$ for some $C>0$. Since $\bigl|~|1-te^{i\theta}|^{-1-i\gamma}~\bigr| \leq t^N|1-te^{i\theta}|^{-1}$, Lemma \ref{sawlem2} follows.
\end{proof}

\subsection{Proof of Theorem \ref{mainthm2}} Choose $\Psi_j \in C^\infty(\Omega)$ in such a way that $0 \leq \Psi_j \leq 1$, $\Psi_j(x)=1$ for $|x|>\frac{2}{j}$, $\Psi_j(x)=0$ for $|x|<\frac{1}{j}$, $|\nabla \Psi_j(x)| \leq c'j$, $|\Delta \Psi_j(x)| \leq c'j^2$.

\begin{proposition} 
\label{fourestlem}
Let $\tau(V,0,\rho)<\infty$. There exists a constant $C=C(\rho,\delta,d)>0$ such that for all positive integers $N$ and $j$
$$
%\label{Aest1}
\|\mathbf{1}_{B(\rho)} \Psi_j |V|^{\frac{1}{2}}[(-\Delta)^{-1}]_{N,N_d^{\delta}}|V|^{\frac{1}{2}} \Psi_j \mathbf{1}_{B(\rho)}\|_{2 \mapsto 2} \leq C \tau(V,0,\rho)^{\frac{1}{d-1}},
\leqno (E1)
$$
$$
%\label{Aest2}
\|\mathbf{1}_{B(\rho)} \Psi_j |V|^{\frac{1}{2}}[(-\Delta)^{-1}]_{N,N_d^{\delta}}|V|^{\frac{1}{2}} \mathbf{1}_{B(3\rho \setminus \rho)}\|_{2 \mapsto 2} \leq C \tau(V,0,3\rho)^{\frac{1}{d-1}},
\leqno (E2)
$$
$$
%\label{Aest4}
\left\|\mathbf{1}_{B(\rho)}\Psi_j |V|^{\frac{1}{2}} [(-\Delta)^{-1}]_{N,N_d^\delta}\mathbf{1}_{B(\frac{2}{j} \setminus \frac{1}{j})} \right\|_{p \mapsto 2} \leq C \tau(V,0,\rho)^{\frac{1}{d-1}}, 
\leqno (E3)
$$
$$
%\label{Aest5}
\left\|\mathbf{1}_{B(\rho)}\Psi_j |V|^{\frac{1}{2}} [(-\Delta)^{-1}]_{N,N_d^\delta}\mathbf{1}_{B(3\rho \setminus 2\rho)} \right\|_{p \mapsto 2} \leq C \tau(V,0,3\rho)^{\frac{1}{d-1}}, 
\leqno (E4)
$$
where $p=\frac{2d}{d+2}$.
\end{proposition}

We prove Proposition \ref{fourestlem} at the end of this section.

\begin{proof}[Proof of Theorem \ref{mainthm2}]
We use the same notations as in the proof of Theorem \ref{mainthm}. Suppose that $u \in Y_V^{\strong}$ satisfies (\ref{diffineq}) and vanishes to an infinite order at $0 \in \Omega$.
We wish to obtain an estimate of the form
\begin{equation}
\label{reqineq7}
\left\|\mathbf{1}_{B(\rho)}\frac{\varphi_{N_d^\delta}}{\varphi_{N_d^\delta}(\rho)} u\right\|_2 \leq C.
\end{equation}
Then, letting $N \to \infty$, we would derive the required identity: $u \equiv 0$ in $B(0,\rho)$. 

The same argument as in the proof of Theorem \ref{mainthm} leads us to an identity
\begin{equation*}
u_{\eta_j}=(-\Delta)^{-1}(-\Delta u_{\eta_j}), \quad \eta_j=\eta \Psi_j,
\end{equation*}
which, in turn, implies
\begin{multline}
\notag
\mathbf{1}_{B(\rho)} \Psi_j V_1^{\frac{1}{2}}\varphi_{N_d^{\delta}}u=\\=
\mathbf{1}_{B(\rho)} \Psi_j V_1^{\frac{1}{2}}[(-\Delta)^{-1}]_{N,N_d^{\delta}}V_1^{\frac{1}{2}}\varphi_{N_d^{\delta}} 
\frac{-\eta_j\Delta u}{V_1^{\frac{1}{2}}}
+\mathbf{1}_{B(\rho)} \Psi_j V_1^{\frac{1}{2}}[(-\Delta)^{-1}]_{N,N_d^{\delta}}\varphi_{N_d^{\delta}} E_j(u).
\end{multline}
Letting $I$ to denote the left hand side of the previous
identity, and, respectively, $I_1$
and $I_2$ the two summands of the right hand side, we rewrite the latter as
\begin{equation*}
I=I_1+I_2.
\end{equation*}
Here $0<\delta<1/2$ is fixed, $2/j \leq \rho$, $\Delta u_{\eta_j}=\eta_j \Delta u+E_j(u)$ and
\begin{equation*}
E_j(u):=2 \nabla \eta_j \nabla u+(\Delta \eta_j)u.
\end{equation*}
Note that $I \in L^2$, since $H^{1,p}_{\loc}(\Omega) \subset X_2$ by Sobolev embedding theorem, and $|V|^{\frac{1}{2}}u \in X_2$ by the definition of $Y_V^{\strong}$. 

Next, we expand $I_1$ as a sum $I_{11}+I_{11}^c$, where
\begin{equation*}
I_{11}:=\mathbf{1}_{B(\rho)} \Psi_j V_1^{\frac{1}{2}}[(-\Delta)^{-1}]_{N,N_d^{\delta}}V_1^{\frac{1}{2}}\mathbf{1}_{B(\rho)}\varphi_{N_d^{\delta}} 
\frac{-\Psi_j\Delta u}{V_1^{\frac{1}{2}}}
\end{equation*}
and
\begin{equation*}
I_{11}^c:=\mathbf{1}_{B(\rho)} \Psi_j V_1^{\frac{1}{2}}[(-\Delta)^{-1}]_{N,N_d^{\delta}}V_1^{\frac{1}{2}}\mathbf{1}_{B(\rho)}^c\varphi_{N_d^{\delta}} 
\frac{-\eta\Delta u}{V_1^{\frac{1}{2}}}.
\end{equation*}
%We estimate terms $I_{11}$ and $I_{11}^c$ by means of 
Proposition \ref{fourestlem} and inequalities (E1) and (E2) imply the required estimates:
\begin{equation*}
%\label{I11est}
\|I_{11}\|_2 \leq C \tau(V_1,0,\rho)^{\frac{1}{d-1}}\|I\|_2
\end{equation*}
and
\begin{equation*}
\|I_{11}^c\|_2 \leq C \varphi_{N_d^\delta}(\rho)\tau(V_1,0,3\rho)^{\frac{1}{d-1}}\|\mathbf{1}_{B(3\rho)}|V|^{\frac{1}{2}}u\|_2.
\end{equation*}

Finally, we represent $I_2$ as a sum $I_{21}+I_{22}$, where
\begin{equation*}
I_{21}:=\mathbf{1}_{B(\rho)}\Psi_j V_1^{\frac{1}{2}}[(-\Delta)^{-1}]_{N,N_d^\delta}\mathbf{1}_{B(\frac{2}{j} \setminus \frac{1}{j})} \varphi_{N_d^\delta} E_j^{(1)}(u)
\end{equation*}
and
\begin{equation*}
I_{22}:=\mathbf{1}_{B(\rho)}\Psi_j V_1^{\frac{1}{2}}[(-\Delta)^{-1}]_{N,N_d^\delta}\mathbf{1}_{B(3\rho \setminus 2\rho)} \varphi_{N_d^\delta} E_j^{(2)}(u).
\end{equation*}
Here
\begin{equation*}
E_j^{(1)}(u):=-2 \nabla \Psi_j \nabla u-(\Delta \Psi_j)u, \quad
E_j^{(2)}(u):=-2 \nabla \eta \nabla u-(\Delta \eta)u.
\end{equation*}
In order to derive an estimate on $\|I_{21}\|_2$, we expand
\begin{equation*}
I_{21}=I_{21}'+I_{21}'',
\end{equation*}
where
\begin{equation*}
I_{21}':=\mathbf{1}_{B(\rho)}\Psi_j V_1^{\frac{1}{2}}[(-\Delta)^{-1}]_{N,N_d^\delta}\mathbf{1}_{B(\frac{2}{j} \setminus \frac{1}{j})} \varphi_{N_d^\delta} (-\Delta \Psi_j)u,
\end{equation*}
\begin{equation*}
I_{21}'':=\mathbf{1}_{B(\rho)}\Psi_j V_1^{\frac{1}{2}}[(-\Delta)^{-1}]_{N,N_d^\delta}\mathbf{1}_{B(\frac{2}{j} \setminus \frac{1}{j})} \varphi_{N_d^\delta} \bigl(-2 \nabla \eta \nabla u\bigr).
\end{equation*}
%We estimate terms $I_{21}'$ and $I_{21}''$:

1) Term $I_{21}'$ presents no problem: by (E3), 
\begin{multline}
\notag
\|I_{21}'\|_2 \leq \left\|\mathbf{1}_{B(\rho)}\Psi_j V_1^{\frac{1}{2}} [(-\Delta)^{-1}]_{N,N_d^\delta}\mathbf{1}_{B(\frac{2}{j} \setminus \frac{1}{j})}\right\|_{p \mapsto 2} \left\|\mathbf{1}_{B(\frac{2}{j} \setminus \frac{1}{j})}\varphi_{N_d^\delta}(\Delta \Psi_j)u\right\|_2 \leq \\ \leq C \tau(V_1,0,\rho)^{\frac{1}{d-1}} \left\|\mathbf{1}_{B(\frac{2}{j} \setminus \frac{1}{j})}\varphi_{N_d^\delta}(\Delta \Psi_j)u\right\|_2,
\end{multline}
where
\begin{equation*}
\left\|\mathbf{1}_{B(\frac{2}{j} \setminus \frac{1}{j})}\varphi_{N_d^\delta}(\Delta \Psi_j)u\right\|_2 \leq C j^{N_d^\delta+2} \left\|\mathbf{1}_{B(\frac{2}{j})}u\right\|_2 \to 0 \quad \text{ as } j \to \infty
\end{equation*}
by the definition of 
the SUC property.

2) In order to derive an estimate on $I_{21}''$, we once again use inequality (E3):
\begin{multline}
\notag
\|I_{21}''\|_2 \leq \left\|\mathbf{1}_{B(\rho)}\Psi_j V_1^{\frac{1}{2}} [(-\Delta)^{-1}]_{N,N_d^\delta}\mathbf{1}_{B(\frac{2}{j} \setminus \frac{1}{j})} \right\|_{p \mapsto 2}\|\mathbf{1}_{B(\frac{2}{j})}\varphi_{N_d^\delta} \nabla \Psi_j \nabla u\|_p \leq \\ \leq C \tau(V_1,0,\rho)^{\frac{1}{d-1}}\|\mathbf{1}_{B(\frac{2}{j})}\varphi_{N_d^\delta} \nabla \Psi_j \nabla u\|_p \leq 
\tilde{C} j^{N_d^\delta+1} \|\mathbf{1}_{B(\frac{2}{j})} \nabla u\|_p,
\end{multline}
where $p:=\frac{2d}{d+2}$. We must estimate $\|\mathbf{1}_{B(\frac{2}{j})}\nabla u\|_2$ by $\|\mathbf{1}_{B(\frac{4}{j})}u\|_2$ in order to apply the SUC property.
For this purpose, we make use of the following well known interpolation inequality
\begin{equation*}
\left\|\mathbf{1}_{B(\frac{2}{j})} \nabla u\right\|_p \leq C j^{\frac{d}{p}} \left(C'j^{\frac{d}{2}-1}\left\|\mathbf{1}_{B(\frac{4}{j})}u\right\|_2+j^{\frac{d+6}{2}}\left\|\mathbf{1}_{B(\frac{4}{j})}\Delta u\right\|_r \right),
\end{equation*}
where $r:=\frac{2d}{d+4}$ (see \cite{Maz}). Using differential inequality (\ref{diffineq}), we reduce the problem to the problem of finding an estimate on $\|\mathbf{1}_{B(\frac{4}{j})} Vu\|_r$ in terms of $\|\mathbf{1}_{B(\frac{4}{j})}u\|_2^{\mu}$, $\mu>0$. By H\"{o}lder inequality,
\begin{equation*}
\left\|\mathbf{1}_{B(\frac{4}{j})}Vu\right\|_r \leq \left\|\mathbf{1}_{B(\frac{4}{j})}|V|^{\frac{1}{2}}u\right\|_2^{\frac{2}{d}}\left\|\mathbf{1}_{B(\frac{4}{j})}V\right\|_{\frac{d-1}{2}}^{\frac{d-1}{d}}\left\|\mathbf{1}_{B(\frac{4}{j})}u\right\|_2^{1-\frac{2}{d}},
\end{equation*}
as required.

As the last step of the proof, we use inequality (E4) to derive an estimate on term $I_{22}$:
\begin{equation*}
\|I_{22}\|_2 \leq C \tau(V_1,0,3\rho)^{\frac{1}{d-1}}\varphi_{N_d^\delta}(\rho)\left\|E_j^{(2)}(u)\right\|_{p}.
\end{equation*}
This estimate and the estimates obtained above imply (\ref{reqineq7}).
\end{proof}

\begin{proof}[Proof of Proposition \ref{fourestlem}]
Estimates (E1) and (E2) follow straightforwardly from Proposition \ref{ourlem}. In order to prove estimate (E3), we introduce the following interpolation function:
\begin{equation*}
F_1(z):=\mathbf{1}_{B(\rho)} \Psi_j |V|^{\frac{d-1}{4}z}\varphi_{N+(\frac{d}{2}-\delta)(1-z)}\left[(-\Delta)^{-\frac{d-1}{2}z}\right]_N \varphi^{-1}_{N+(\frac{d}{2}-\delta)(1-z)} \mathbf{1}_{B(\frac{2}{j} \setminus \frac{1}{j})}, \quad 0 \leq \Real(z) \leq 1.
\end{equation*}
According to Lemma \ref{JKlem}, $\|F_1(i\gamma)\|_{2 \mapsto 2} \leq C_1 e^{c_1|\gamma|}$ for appropriate $C_1$, $c_1>0$. Further, according to Lemma \ref{sawlem2},
\begin{multline}
\notag
\|F_1(1+i\gamma)\|_{\frac{2d}{2d-1} \mapsto 2} \leq C_2 e^{c_2 \gamma^2}\left\|\mathbf{1}_{B(\rho)}|V|^{\frac{d-1}{4}}(-\Delta)^{-\frac{d-1}{2}}\right\|_{\frac{2d}{2d-1} \mapsto 2} \leq \\
\leq C_2 e^{c_2 \gamma^2}\left\|\mathbf{1}_{B(\rho)}|V|^{\frac{d-1}{4}}(-\Delta)^{-\frac{d-1}{4}}\right\|_{2 \mapsto 2}\left\|(-\Delta)^{-\frac{d-1}{4}}\right\|_{\frac{2d}{2d-1} \mapsto 2} \leq \\ \leq C_2 e^{c_2 \gamma^2}\tau(V,x_0,\rho)^{\frac{1}{2}}\left\|(-\Delta)^{-\frac{d-1}{4}}\right\|_{\frac{2d}{2d-1} \mapsto 2}
\end{multline}
for appropriate $C_2$, $c_2>0$, where, clearly, $\|(-\Delta)^{-\frac{d-1}{4}}\|_{\frac{2d}{2d-1} \mapsto 2}<\infty$. Therefore, by Stein's interpolation theorem,
\begin{equation*}
\left\|F_1\left(\frac{2}{d-1}\right)\right\|_{p \mapsto 2} \leq C\tau(V,x_0,\rho)^{\frac{1}{2(d-1)}}.
\end{equation*}
The latter inequality implies (E3). 

The proof of estimate (E4) is similar: it suffices to consider interpolation function
\begin{equation*}
F_2(z):=\mathbf{1}_{B(\rho)} \Psi_j |V|^{\frac{d-1}{4}z}\varphi_{N+(\frac{d}{2}-\delta)(1-z)}\left[(-\Delta)^{-\frac{d-1}{2}z}\right]_N \varphi^{-1}_{N+(\frac{d}{2}-\delta)(1-z)} \mathbf{1}_{B(3\rho \setminus 2\rho)}
\end{equation*}
for $0 \leq \Real(z) \leq 1$.
\end{proof}

\section{Proof of Theorem \ref{sawthm}}

\label{exsect}

\begin{proof}[Proof of Theorem \ref{sawthm}]
Let $u \in Y_V^{\mathcal K}$. Suppose that $u \equiv 0$ in some neighbourhood of $0$. Assume that $\rho>0$ is sufficiently small, so that $\bar{B}(0,2\rho) \subset \Omega$, and let $\eta \in C^\infty(\Omega)$ be such that $\eta \equiv 1$ on $B(0,\rho)$, $\eta \equiv 0$ on $\Omega \setminus B(0,2\rho)$. We may assume, without loss of generality, that $V \geq 1$. The standard limiting argument implies the following identity:
\begin{equation*}
\mathbf{1}_{B(\rho)} u=\mathbf{1}_{B(\rho)} [(-\Delta)^{-1}]_{N} (-\Delta u_\eta).
\end{equation*}
Therefore, we can write
\begin{multline}
\notag
%\label{sawrepr}
\mathbf{1}_{B(\rho)} \varphi_N V u=\\=\mathbf{1}_{B(\rho)} \varphi_N V [(-\Delta)^{-1}]_N \varphi_{N}^{-1} \mathbf{1}_{B(\rho)} \varphi_N (-\Delta u)+\mathbf{1}_{B(\rho)} \varphi_N V [(-\Delta)^{-1}]_N \varphi_N^{-1}\mathbf{1}^c_{B(\rho)}\varphi_N (-\Delta u_\eta),
\end{multline}
or, letting  $K$  to denote the left hand side and,
respectively, $K_1$  and  $K_2$  the two summands of the right hand side of
the last equality, we rewrite the latter as $$K=K_1+K_2.$$
Note that $K \in L^1(\mathbb R^d)$, as follows from definition of space $Y_V^{\mathcal K}$. Lemma \ref{sawlem1} implies that
\begin{equation*}
\|\mathbf{1}_{B(\rho)} \varphi_N V [(-\Delta)^{-1}]_N \varphi_{N}^{-1}  f\|_{1} \leq C\|\mathbf{1}_{B(\rho)} V (-\Delta)^{-1} f\|_{1} \leq C\beta \|f\|_{1},
\end{equation*}
for all $f \in L^1(\Omega)$,
which implies an estimate on $K_1$:
\begin{equation*}
%\|\mathbf{1}_{\rho} \varphi_N V [(-\Delta)^{-1}]_N \varphi_{N}^{-1} \mathbf{1}_{\rho} \varphi_N (-\Delta u)\|_{1} \leq C\beta \|\mathbf{1}_\rho\varphi_N V u \|_{1}.
\|K_1\|_{1} \leq C\beta \|K\|_{1}.
\end{equation*}
In order to estimate $K_2$, we first note that $\mathbf{1}^c_{B(\rho)} (-\Delta u_\eta)=\mathbf{1}_{B(2\rho \setminus \rho)} (-\Delta u_\eta)$. According to Lemma \ref{sawlem1} there exists a constant $\hat{C}>0$ such that
\begin{equation*}
\|\mathbf{1}_{B(2\rho)} \varphi_N V [(-\Delta)^{-1}]_N \varphi_{N}^{-1}\|_{1 \mapsto 1} \leq \hat{C}.
\end{equation*}
Hence,
\begin{equation*}
\|K_2\|_{1} \leq \hat{C}\|\mathbf{1}_{B(2\rho \setminus \rho)} \varphi_N (-\Delta u_\eta)\|_{1} \leq \hat{C}\rho^{-N}\|\Delta u_{\eta}\|_1.
\end{equation*}
Let us choose $\beta>0$ such that $C\beta<1$. Then the estimates above imply
\begin{equation*}
(1-C\beta)\|\mathbf{1}_{B(\rho)} \rho^N \varphi_N u\|_{1} \leq (1-C\beta)\|\rho^N K\|_{1} \leq \|\rho^N K_2\|_{1} \leq \hat{C}\|\Delta u_{\eta}\|_1.
\end{equation*}
%and
%\begin{equation*}
%(1-C\beta)\|\rho^N K\|_{1} \leq \|\rho^N K_2\|_{1} \leq \hat{C}\|\Delta u_{\eta}\|_1.
%\end{equation*}
%Therefore,
%\begin{equation*}
%(1-C\beta)\|\mathbf{1}_{B(\rho)} \rho^N \varphi_N u\|_{1} \leq \hat{C} \|\Delta u_\eta\|_{1}.
%\end{equation*}
Letting $N \to \infty$, we obtain $u \equiv 0$ in $B(0,\rho)$.
\end{proof}

\bibliographystyle{alpha}
\bibliography{unique}

\begin{thebibliography}{FHHOHO82}

\bibitem[ABG]{ABG}
W.O. Amrein, A.M. Bertier, and V.~Georgescu.
\newblock ${L^p}$-inequalities for the {L}aplacian and unique continuation.
\newblock {\em Ann. Inst. Fourier, Grenoble}, 31, 1981.

\bibitem[C]{Carl}
T.~Carleman.
\newblock Sur un probleme d'unicite pour les systemes d'equations aux derives
  partielles a deux variables independantes.
\newblock {\em Ark. Mat.}, 26B:1--9, 1939.

\bibitem[CS]{SawChan}
S.~Chanillo and E.T. Sawyer.
\newblock Unique continuation for {$\Delta+V$} and {C.}{F}efferman-{P}hong
  class.
\newblock {\em Trans. Amer. Math. Soc.}, 318:275--300, 1990.


\bibitem[F]{F}
C.~Fefferman.
\newblock The uncertanty principle.
\newblock {\em Bull. Amer. Math. Soc.}, 9:129--206, 1983.



\bibitem[F3H]{Forese}
R.~Froese, I.~Herbst, M.~Hoffmann-Ostenhof, and T.~Hoffmann-Ostenhof.
\newblock ${L_2}$-exponential lower bounds to solutions of the {S}chrodinger
  equations.
\newblock {\em Commun. Math. Phys.}, 87:265--286, 1982.

\bibitem[JK]{JK}
D.~Jerison and C.E. Kenig.
\newblock Unique continuation and absence of positive eigenvalues for {L}aplace
  operator.
\newblock {\em Ann. of Math.}, 121:463--494, 1985.


\bibitem[K1]{Kato0}
T.~Kato.
\newblock Notes on some inequalities for linear operators.
\newblock {\em Math. Ann.}, 4:208--212, 1952.


\bibitem[K2]{Kato1}
T.~Kato.
\newblock Growth properties of solutions of the reduced wave equation with a
  variable coefficient.
\newblock {\em Commun. Pure. Appl. Math.}, 12:403--425, 1959.

\bibitem[K3]{Kato2}
T.~Kato.
\newblock {\em Perturbation Theory for Linear Operators}.
\newblock Springer-Verlag, 1966.



\bibitem[KiSh]{KiSh}
D.~Kinzebulatov and L.~Shartser.
\newblock Towards an optimal result on unique continuation for eigenfunctions of Schr\"{o}dinger operators.
\newblock {\em To appear in C. R. Math. Rep. Acad. Sci. Canada}, 2009.


\bibitem[KS]{KS}
R.~Kerman and E.T.~Sawyer.
\newblock The trace inequality and eigenvalue estimates for Schr\"{o}dinger operators.
\newblock {\em Annales de l'institut Fourier}, 36:207--228, 1986.

\bibitem[KPS]{KPS}
V.F. Kovalenko, M.A. Perelmuter, and Yu.A. Semenov.
\newblock Schrodinger operators with ${L\sp{1/2}\sb{W}}( R\sp{l})$-potentials.
\newblock {\em J. Math. Phys.}, 22:1033--1044, 1981.



\bibitem[KT]{KT}
H.~Koch and D.~Tataru.
\newblock Sharp counterexamples in unique continuation for second order
  elliptic equations.
\newblock {\em J. reine angew. Math.}, 542:133--146, 2002.

\bibitem[LS]{LS}
V.~Liskevich and Yu.A. Semenov.
\newblock Some problems on {M}arkov semigroups.
\newblock {\em Advances in Partial Differential Equations}, 11:163--217, 1996.

\bibitem[M]{Maz}
V.~Maz'ya.
\newblock {\em Sobolev spaces}.
\newblock Springer-Verlag, 1985.

\bibitem[MS]{MS}
P.D. Milman and Yu.A. Semenov.
\newblock Global heat kernel bounds via desingularizing weights.
\newblock {\em J. Func. Anal.}, 212:273--398, 2004.

\bibitem[RS]{RS}
M.~Reed and B.~Simon.
\newblock {\em Methods of Modern Mathematical Physics II. Fourier Analysis,
  Self-Adjointness}.
\newblock Academic Press, 1975.

\bibitem[RV]{RV}
A. Ruiz and L. Vega
\newblock Unique continuation for {S}chr\"{o}dinger operators 
with potentials in Morrey spaces.
\newblock {\em Publicacions Matematiques}, 35:291--298, 1991.


\bibitem[S]{Saw}
E.T. Sawyer.
\newblock Unique continuation for {S}chr\"{o}dinger operators in dimensions
  three or less.
\newblock {\em Annales de l'institut Fourier}, 34:189--200, 1984.

\bibitem[SS]{SS}
M.~Schechter and B.~Simon.
\newblock Unique continuation for {S}chr\"{o}dinger operators with unbounded
  potentials.
\newblock {\em J. Math. Anal. Appl.}, 77:482--492, 1980.

\bibitem[St1]{SteBook}
E.M. Stein.
\newblock {\em Singular Integrals and Differentiability Properties of
  Functions}.
\newblock Princeton University Press, 1970.

\bibitem[St2]{Ste}
E.M. Stein.
\newblock Appendix to ``{U}nique continuation''.
\newblock {\em Ann. of Math.}, 121:488--494, 1985.



\bibitem[SW]{SW}
E.M. Stein and G.~Weiss.
\newblock {\em Introduction to {Fourier} Analysis on Euclidean Spaces}.
\newblock Princeton University Press, 1971.


\bibitem[W]{W}
T.H.~Wolff.
\newblock Unique continuation for $\vert \Delta u\vert \le V\vert \nabla u\vert $ and related problems.
\newblock {\em Rev. Mat. Iberoamericana}, 6:155--200, 1990.




\end{thebibliography}

\end{document}